\documentclass[a4paper,11pt]{amsart}

\usepackage{amsmath}
\usepackage{amssymb}
\usepackage{amsthm}
\usepackage[all]{xy}
\usepackage{enumerate}

\def\Yj{                  
\begin{picture}(4,4)
\put(0,4){\line(1,0){4}}
\put(0,0){\line(1,0){4}}
\put(0,4){\line(0,-1){4}}
\put(4,4){\line(0,-1){4}}
\end{picture}}

\def\Ydj{                 
\begin{picture}(8,8) 
\put(0,8){\line(1,0){8}}
\put(0,4){\line(1,0){8}}
\put(0,0){\line(1,0){4}}
\put(0,8){\line(0,-1){8}}
\put(4,8){\line(0,-1){8}}
\put(8,8){\line(0,-1){4}}
\end{picture}
}

\def\Ydd{                 
\begin{picture}(8,8) 
\put(0,8){\line(1,0){8}}
\put(0,4){\line(1,0){8}}
\put(0,0){\line(1,0){8}}
\put(0,8){\line(0,-1){8}}
\put(4,8){\line(0,-1){8}}
\put(8,8){\line(0,-1){8}}
\end{picture}
}

\def\Ytjj{                 
\begin{picture}(12,12) 
\put(0,12){\line(1,0){12}}
\put(0,8){\line(1,0){12}}
\put(0,4){\line(1,0){4}}
\put(0,0){\line(1,0){4}}
\put(0,12){\line(0,-1){12}}
\put(4,12){\line(0,-1){12}}
\put(8,12){\line(0,-1){4}}
\put(12,12){\line(0,-1){4}}
\end{picture}
}

\def\Ytdj{                 
\begin{picture}(12,12) 
\put(0,12){\line(1,0){12}}
\put(0,8){\line(1,0){12}}
\put(0,4){\line(1,0){8}}
\put(0,0){\line(1,0){4}}
\put(0,12){\line(0,-1){12}}
\put(4,12){\line(0,-1){12}}
\put(8,12){\line(0,-1){8}}
\put(12,12){\line(0,-1){4}}
\end{picture}
}

\def\Yttd{                 
\begin{picture}(12,12) 
\put(0,12){\line(1,0){12}}
\put(0,8){\line(1,0){12}}
\put(0,4){\line(1,0){12}}
\put(0,0){\line(1,0){8}}
\put(0,12){\line(0,-1){12}}
\put(4,12){\line(0,-1){12}}
\put(8,12){\line(0,-1){12}}
\put(12,12){\line(0,-1){8}}
\end{picture}
}

\def\Yttt{                 
\begin{picture}(12,12) 
\put(0,12){\line(1,0){12}}
\put(0,8){\line(1,0){12}}
\put(0,4){\line(1,0){12}}
\put(0,0){\line(1,0){12}}
\put(0,12){\line(0,-1){12}}
\put(4,12){\line(0,-1){12}}
\put(8,12){\line(0,-1){12}}
\put(12,12){\line(0,-1){12}}
\end{picture}
}

\def\Ycd{                 
\begin{picture}(16,8) 
\put(0,8){\line(1,0){16}}
\put(0,4){\line(1,0){16}}
\put(0,0){\line(1,0){8}}
\put(0,8){\line(0,-1){8}}
\put(4,8){\line(0,-1){8}}
\put(8,8){\line(0,-1){8}}
\put(12,8){\line(0,-1){4}}
\put(16,8){\line(0,-1){4}}
\end{picture}
}

\newtheorem{thm}{Theorem}[section]
\newtheorem{lemma}[thm]{Lemma}

\newtheorem{ex}[thm]{Example}
\newtheorem{prop}[thm]{Proposition}
\numberwithin{equation}{section}

\newcommand{\xgr}{\mathfrak{gr}}
\newcommand{\lag}{\mathfrak g}
\newcommand{\lap}{\mathfrak p}

\newcommand{\xJ}{\mathfrak J}
\newcommand{\xS}{\mathfrak S}

\newcommand{\laso}{\mathfrak{so}}

\newcommand{\lagl}{\mathfrak{gl}}

\newcommand{\xj}{\mathfrak{j}}

\newcommand{\mV}{\mathbb V}

\newcommand{\mW}{\mathbb W}
\newcommand{\mE}{\mathbb E}
\newcommand{\mF}{\mathbb F}

\newcommand{\gP}{\mathrm P}

\newcommand{\GL}{\mathrm{GL}}

\newcommand{\Spin}{\mathrm{Spin}}
\newcommand{\G}{\mathrm G}

\newcommand{\R}{\mathbb R}
\newcommand{\N}{\mathbb N}
\newcommand{\C}{\mathbb C}
\newcommand{\Sp}{\mathbb S}
\newcommand{\T}{\mathbb T}

\newcommand{\Z}{\mathbb Z}

\newcommand{\cE}{\mathcal E}
\newcommand{\cO}{\mathcal O}

\newcommand{\cC}{\mathcal C}

\newcommand{\cJ}{\mathcal J}
\newcommand{\cA}{\mathcal A}

\newcommand{\cU}{\mathcal U}

\newcommand{\ra}{\rightarrow}

\newcommand{\Ad}{\operatorname{Ad}}

\title{Resolution of the $k$-Dirac operator}
\author{Tom\'a\v s Sala\v c }
\thanks{
 The  research was partially supported by  the grant 17-01171S of the Grant Agency of the Czech Republic.\\
}

\begin{document}

\begin{abstract}
This is the second part in a series of two papers. The $k$-Dirac complex is a complex of differential operators which are naturally associated to a particular $|2|$-graded parabolic geometry. In this paper we will consider the $k$-Dirac complex over the homogeneous space of the parabolic geometry and  as a first result, we will prove that the $k$-Dirac complex is formally exact (in the sense of formal power series). Then we will  show that the $k$-Dirac complex  descends from an affine subset of the homogeneous space to a complex of linear and constant coefficient differential operators and  that the first operator in the descended complex is  the $k$-Dirac operator studied in Clifford analysis. The main result of this paper is that the descended complex is locally exact and thus it forms a resolution of the $k$-Dirac operator. 
\end{abstract}

\maketitle

\section{Introduction}
Let $\{\varepsilon_1,\ldots,\varepsilon_{2n}\}$ be the standard basis of $\R^{2n}$, $B$ be the standard inner product, $\Sp$ be the complex space of spinors (see \cite[Section 6]{GW}) of the complexified Clifford algebra of $(\R^{2n},B)$ and $U:=M(2n,k,\R)$ be the vector space of matrices of size $2n\times k$ with real coefficients. We we will use matrix coefficients $x_{\alpha i},\ \alpha=1,\dots,2n,\ i=1,\dots,k$ as coordinates on $U$ and we denote by $\partial_{x_{\alpha i}}$ the  coordinate vector fields. Let $\cC^\infty(U,\mV)$ be the space of  smooth functions on $U$ with values in a vector space $\mV$.
The differential operator 
\begin{align}\label{k-Dirac operator}
 \underline D_0&:\cC^\infty(U,\Sp)\ra\cC^\infty(U,\C^k\otimes_\C\Sp),&\\
 \underline D_0\psi&=\sum_{\alpha=1}^{2n}(\varepsilon_\alpha.\partial_{x_{\alpha 1}}\psi,\ \dots,\ \varepsilon_\alpha.\partial_{x_{\alpha k}}\psi)&\nonumber
\end{align}
is known (see  \cite{CSSS} and \cite{SSSL}) as   the $k$-\textit{Dirac operator}. Here $\varepsilon_\alpha.$ is the usual action of $\varepsilon_\alpha$ on  $\Sp$ and we view $\C^k\otimes_\C\Sp$ as the vector space of $k$-tuples of spinors.  The  operator  generalizes  the $k$-Cauchy-Riemann operator just as the Dirac operator can be viewed as  a generalization of the Cauchy-Riemann operator.  We assume  $k\ge2$ throughout the article. 

The $k$-Dirac operator is an overdetermined, linear and first order differential operator with constant coefficients and so it is natural to look for a resolution of this operator, i.e. to look for a sequence of differential operators which starts with the $k$-Dirac operator and which is locally exact (loosely speaking the image of each operator in the sequence coincides with the kernel of the next operator on a sufficiently small neighborhood of any point $x\in U$). In other words, the associated  sequence of sheaves is exact. As a potential application, there is an open problem of characterizing the domains of monogenicity, i.e. an open set $\cU$ is a domain of monogenicity if at each boundary point there is a null solution of (\ref{k-Dirac operator}) which is defined on $\cU$ which cannot by continued beyond the boundary point by a null solution. Recall \cite[Section 4]{H}  that the Dolbeault resolution together with some $L^2$ estimates are crucial in a proof of the statement that any pseudoconvex domain is a domain of holomorphy.

If\footnote{This condition  is known as the stable range (see \cite{CSSS} and \cite{SSSL}).}  $n\ge k$, then  we will show (see Theorem \ref{the main theorem}) that such a resolution is obtained by descending the $k$-\textit{Dirac complex}. The $k$-Dirac complex (see \cite{S} and  \cite{TS})  is a complex of linear differential operators which are naturally associated   to a certain parabolic geometry of type $(\G,\gP)$ where  $\G:=\Spin(k,2n+k)$ and $\gP$ is a parabolic subgroup associated to a $|2|$-grading on the Lie algebra  $\lag=\laso(k,2n+k)$ of $\G$. We will consider here the $k$-Dirac complex only over the homogeneous model $\G/\gP$.

From the point of view of the representation theory, the $k$-Dirac  complex belongs to so called singular central (or infinitesimal) character  and so this complex does not come from the BGG machinery introduced in \cite{CSlS}. It is explained in \cite{S}  that the $k$-Dirac complex can be constructed using the machinery of the Penrose transform which means that the $k$-Dirac complex is the direct image of a relative BGG complex. Even though this might seem a bit clumsy when one is interested only in proving the existence of the complex, the advantage of this approach is that  one  can show (see \cite[Theorem 7.14]{S}) that the $k$-Dirac complex is formally exact in the sense of  \cite{Sp}, i.e. it induces a long exact sequence  of infinite weighted jets  at any  fixed point. 

The Penrose transform (see \cite{BE}) uses powerful tools of sheaf theory in the realm of complex manifolds, in particular  the Bott-Borel-Weil theorem is crucial. This brings in one technical point, namely  in \cite{S} we construct  the $k$-Dirac complex over the homogeneous space $\G^\C/\gP^\C$ of a complex parabolic geometry of type $(\G^\C,\gP^\C)$ where $\G^\C:=\Spin(2m,\C),\ m=k+n$ and $\gP^\C$ is a parabolic subgroup that is associated to the "complexified" $|2|$-grading on $\lag_\C:=\lag\otimes\C$. However, we will observe in Section \ref{section embedding} that a linear and $\G^\C$-invariant operator on $\G^\C/\gP^\C$ induces a linear and $\G$-invariant differential operator on $\G/\gP$ and thus, the complex from \cite{S} induces the $k$-Dirac complex on $\G/\gP$. This passage is  also in Introduction of \cite{B} where it is explained that the $k$-Cauchy-Fueter operator (see for example \cite{BAS} and \cite{CSS}) can  be viewed as the first  operator  in one of the quaternionic complexes which lives on the Grassmannian of complex $2$-planes in $\C^{2k+2}$.

Let us also mention that the assumption $n\ge k$ is not  needed in this article but only in the first part of the series \cite{S}. However, it is plausible (see  \cite{K}) that the machinery of the Penrose transform works also in the case $n<k$. Unfortunately, due to the representation theory, the Penrose transform does not work for the $k$-Dirac operator in dimension $2n+1$, even though (see \cite{F}), there are $k$-Dirac complexes also in odd dimensions but it is not clear whether these complexes are formally exact.

Viewing a differential operator as the direct image of another differential operator is not very useful when one is interested in local formulas. The differential operators in the $k$-Dirac complex are of  first and second order. While the first order operators can be handled rather easily (see  \cite{SS}), dealing with  second order operators is  more difficult. Local formulas for the second order operators in the $k$-Dirac complex were given  in \cite{TS}. 

As we mentioned above, the resolution of (\ref{k-Dirac operator}) is obtained by descending the $k$-Dirac complex. The descending of differential operators which are natural to parabolic structures  was developed in the recent series of papers  \cite{CSaI}, \cite{CSaII} and \cite{CSaIII}  with preliminary paper \cite{CSa} in full generality for parabolic contact structures.  The parabolic geometry of type $(\G,\gP)$ is contact if, and only if $k=2$. Thus, for $k>2$ we have to consider a higher dimensional analogue of the construction. Nevertheless, as we will work only on an open, dense and affine subset of $\G/\gP$ which is known as the "big cell", the descending procedure will turn out to be rather easy and we will not need a general theory. 

Finally, let us summarize the main results of this article. We will fix a "canonical" trivialization of the canonical $\gP$-principal bundle over the big cell so that we can view sections of associated vector bundles as vector valued functions and operators in the $k$-Dirac complex as differential operators with polynomial coefficients. In this picture, it is standard to talk about formal power series and we will show in  Theorem \ref{thm exactnes with formal power series} that the formal exactness of the $k$-Dirac complex implies the exactness of the complex with formal power series at any fixed point.   In Theorem \ref{the main theorem} we will prove that the descended complex is a resolution of the $k$-Dirac operator.

\bigskip

Acknowledgement. The author is grateful to Vladim\'ir Sou\v cek for his support and many useful conversations. The author would also like to thank to Luk\'a\v s Krump for the possibility of using his package for Young diagram.  The author wishes to thank to the unknown referee for many helpful suggestions which considerably improved and simplified the current manuscript. 

\begin{center}
\textbf{ Notation}
\end{center}
$M(n,k,\T)$ matrices of size $n\times k$ with coefficients from a field $\T$\\
$M(n,\T):=M(n,n,\T)$\\
$A(n,\T):=\{A\in M(n,\T)|\ A^T=-A\}$\\
$1_n=$  identity  $n\times n$-matrix\\
$[v_1,\dots,v_\ell]=$ linear span of vectors $v_1,\dots,v_\ell$\\
$L(\mV,\mW)$ the space of linear maps $\mV\ra\mW$

\section{The parabolic geometry of type $(\G,\gP)$}\label{section parabolic geometry}
We will define in Section \ref{section lie algebra} the $|2|$-grading on $\lag$  which determines the parabolic subgroup $\gP$ from Introduction. We will recall in Section \ref{section homogeneous space} some well known theory of parabolic geometries, namely we will need some basic properties of invariant differential operators over the homogeneous space $\G/\gP$ and the filtration of the tangent bundle associated to any $|2|$-graded parabolic geometry. In Section \ref{section affine subset} we will move to the open, dense and affine subset of $\G/\gP$ which is also called the big cell. We  will identify the big cell with a Lie group and fix a trivialization of the $\gP$-principal bundle over the big cell  which leads to several simplifications. In Section \ref{section invariant op and wjets} we will recall some well known theory of invariant differential operators  over the big cell and recall the notion of weighted jets. In Section \ref{section descending natural operators} we will explain that a linear and invariant operator on the big cell descends to a linear and constant coefficient operator on the affine set $U= M(2n,k,\R)$. In Section \ref{section embedding} we will consider the complex parabolic geometry of type $(\G^\C,\gP^\C)$ from  Introduction.

A comprehensive introduction into the theory of parabolic geometries is \cite{CS}. The concept of weighted jets was originally introduced by Tohru Morimoto, see for example \cite{Mo} or \cite{MoI}. As we will work only over the big cell, we will not give here the definition of weighted jets over a general filtered manifold as over the big cell (which we view as a Lie group) we can define the weighted jets via left invariant vector fields. 

\subsection{Lie algebra  $\lag$}\label{section lie algebra}
Let $\{e_1,\ldots,e_k,\varepsilon_1,\ldots,\varepsilon_{2n},e^1,\ldots,e^k\}$ be the standard basis of $\R^{2m}$ where $m=n+k$ and $h$ be the symmetric bilinear form such that 
$$h(e_i,e^j)=\delta_{ij},\ h(e_i,\varepsilon_\alpha)=h(e^i,\varepsilon_\alpha)=0,\ h(\varepsilon_\alpha,\varepsilon_\beta)=-\delta_{\alpha\beta}$$
where $i,j=1,\dots,k;\ \alpha,\beta=1,\dots,2n$ and $\delta$ is the Kronecker delta. Notice that the associated quadratic form is non-degenerate with signature $(k,2n+k)$. We will sometimes write $\R^{k,2n+k}$ to indicate that we consider on $\R^{2m}$ the bilinear form $h$. 
The associated Lie algebra $\lag:=\laso(h)\cong\laso(k,2n+k)$ is 
\begin{equation}\label{orthogonal Lie alg}
\Bigg\{\left(\begin{array}{ccc}
A&Z^T&W\\
X&B&-Z\\
Y&-X^T&-A^T
\end{array}\right)\Bigg|\begin{matrix}
A\in M(k,\R),Y,W\in A(k,\R),\\
B\in A(2n,\R),X,Z\in M(2n,k,\R)
\end{matrix}\Bigg\}.
\end{equation}

The block decomposition  determines a direct sum decomposition $\lag=\lag_{-2}\oplus\lag_{-1}\oplus\lag_0\oplus\lag_1\oplus\lag_2$ where we put
\begin{eqnarray}\label{gradation on g}
&\lag_{-2}:=\Bigg\{
\left(\begin{matrix}
0&0&0\\
0&0&0\\
\ast&0&0\\
\end{matrix}
\right)
\Bigg\},\
\lag_{-1}:=\Bigg\{
\left(\begin{matrix}
0&0&0\\
\ast&0&0\\
0&\ast&0\\
\end{matrix}
\right)
\Bigg\},\ 
\lag_{0}:=\Bigg\{
\left(\begin{matrix}
\ast&0&0\\
0&\ast&0\\
0&0&\ast\\
\end{matrix}
\right)
\Bigg\},\nonumber&\\ 
&\lag_{1}:=\Bigg\{
\left(\begin{matrix}
0&\ast&0\\
0&0&\ast\\
0&0&0\\
\end{matrix}
\right)
\Bigg\},\ 
\lag_{2}:=\Bigg\{
\left(\begin{matrix}
0&0&\ast\\
0&0&0\\
0&0&0\\
\end{matrix}
\right)
\Bigg\}&
\end{eqnarray}

It is straightforward to verify that 
\begin{enumerate}[(I)]
\item $[\lag_i,\lag_j]\subset\lag_{i+j}$ where $i,j\in\Z$ and  we agree that $\lag_i=\{0\}$ if $|i|>2$ and
\item $\lag_{-1}$ generates $\lag_-:=\lag_{-2}\oplus\lag_{-1}$ as a Lie algebra.
\end{enumerate}
Hence, the direct sum decomposition is a $|2|$-grading (see \cite[Definition 3.1.2]{CS}). The filtration associated to the grading is $\lag=\lag^{-2}\supset\lag^{-1}\supset\dots\supset\lag^{2}\supset\{0\}$ where we put $\lag^i:=\bigoplus_{j\ge i}\lag_j$. From the property (I) follows that $\lag_-,\ \lag_0$ and $\lag^i$ where $i\ge0$ are subalgebras  of $\lag$ and that  $\lag_-$ and $\lag^i,\ i\ge1$ are nilpotent. We call $\lap:=\lag^0$ the \textit{parabolic subalgebra} associated to the grading and put $\lap_+:=\lag^1$. Then  $\lap=\lag_0\oplus\lap_+$ is the Levi-decomposition of $\lap$ which means that $\lag_0$ is a maximal reductive subalgebra of $\lap$ called  \textit{the reductive Levi factor} and that $\lap_+$ is the nilradical of $\lap$ (see \cite[Section 2.1.8]{CS}).  

The Killing form  induces (see   \cite[Proposition 3.1.2]{CS}) isomorphisms  $\lag_i\cong\lag_{-i}^\ast$ of $\lag_0$-modules where ${}^\ast$ denotes the dual module and $i=-2,\dots,2$. By the same Proposition,  there is a \textit{grading element}  $E\in\lag$  which  is uniquely determined by   $\lag_i=\{X\in\lag:[E,X]=iX\},\ i=-2,-1,\dots,2$. The grading element associated to (\ref{gradation on g}) corresponds to the block matrix  from (\ref{orthogonal Lie alg}) where $A=1_k$ and all the other matrices  are  zero. 	

The algebra $\lag_0$ is isomorphic to $\lagl(k,\R)\oplus\laso(2n)$. Assume that $\mE$ and $\mF$ are $\lagl(k,\R)$ and $\laso(n)$-modules, respectively. Then $\mE$ and $\mF$ are also $\lag_0$-modules by letting the other factor act trivially. We denote by $\mE\boxtimes\mF$ the exterior tensor product of $\mE$ and $\mF$. Notice that the subspaces $[e_1,\dots,e_k]$ and $[\varepsilon_1,\dots,\varepsilon_{2n}]$ are $\lag_0$-invariant and that they are isomorphic to the $\lag_0$-modules $\R^k$ and $\R^{2n}$, respectively. Now it is easy to see  that there are isomorphisms of $\lag_0$-modules:
\begin{equation}\label{g0 pieces in lap}
\lag_{-2}\cong\Lambda^{2}\R^{k\ast},\  \lag_{-1}\cong\R^{k\ast}\boxtimes\R^{2n},\  \lag_1\cong\R^k\boxtimes\R^{2n}\ \ \mathrm{and}
 \ \ \lag_2\cong\Lambda^{2}\R^k
\end{equation}
where we use the $\lag_0$-invariant inner product $h|_{\R^{2n}}$.
 The standard basis of $\R^{2m}$ gives the bases
\begin{equation}\label{bases of g-}
\{e^r\wedge e^s:\ 1\le r<s\le k\} \ \mathrm{and}\ \{e^i\otimes\varepsilon_\alpha:\ i=1,\dots,k,\ \alpha=1,\dots,{2n}\}
\end{equation}
of $\lag_{-2}$ and  $\lag_{-1}$, respectively, and the dual bases 
\begin{equation}\label{dual basis}
\{e_{r}\wedge e_{s}:\ 1\le r<s\le k\}\ \mathrm{and}\ \{e_i\otimes\varepsilon_{\alpha}:\ i=1,\dots,k,\ \alpha=1,\dots,n\}
\end{equation}
of $\lag^2$ and $\lag^1$, respectively. 

  
Finally, let us consider the Lie bracket $\Lambda^2\lag_-\ra\lag_-$. This map is homogeneous of degree zero and thus, we can look at each homogeneous part separately. It is the clear that the map  is non-zero only in the homogeneity -2. Moreover, by the Jacobi identity, the map is $\lag_0$-equivariant. If we use the isomorphisms from (\ref{g0 pieces in lap}), then it is easy to see that there is (up to constant) a unique $\lag_0$-equivariant map which is the composition   
\begin{equation}\label{lie bracket on g-}
\Lambda^2\lag_{-1}\cong\Lambda^2(\R^{k\ast}\boxtimes\R^{2n})\ra\Lambda^2\R^{k\ast}\boxtimes S^2\R^{2n}\ra\Lambda^2\R^{k\ast}\cong\lag_{-2}
\end{equation}
where the first map is the obvious projection and in the second map  we  take the trace with respect to $h|_{\R^{2n}}$. Notice that the map (\ref{lie bracket on g-}) is non-degenerate.

\subsection{Lie groups, homogeneous space and invariant differential operators}\label{section homogeneous space}
Let $\G:=\Spin(k,2n+k)$ and  $\Ad:\G\ra\GL(\lag)$ be the adjoint representation.  We call
\begin{equation}\label{parabolic subgroup}
\gP:=\{g\in\G: \Ad(g)(\lag^i)\subset\lag^i;\ i=-2,\dots2\}
\end{equation}
the \textit{parabolic subgroup} associated to the 2-grading and 
\begin{equation}\label{Levi factor}
 \G_0:=\{g\in\G: \Ad(g)(\lag_i)\subset\lag_i;\ i=-2,\dots2\}
\end{equation}
the \textit{Levi subgroup} of $\gP$. Then $\gP$ is a closed subgroup of $\G$ with Lie algebra  $\lap$ and  $\G_0$ is a closed subgroup of $\gP$ with Lie algebra $\lag_0$. It can be shown  (see \cite{TSIII})  that $\G_0$ is isomorphic to $\GL(k,\R)\times\Spin(2n)$. We put $\gP_+:=\exp(\lap_+)$.
By \cite[Theorem 3.1.3]{CS}, the map $\exp:\lap_+\ra\gP_+$ is a diffeomorphism, $\gP_+$ is a normal subgroup of $\gP$ and  $\gP=\G_0\ltimes\gP_+$. Hence, there is a surjective Lie group homomorphism $\gP\ra\G_0$ and thus, any $\G_0$-module is  also a $\gP$-module with trivial action of  $\gP_+$. It is well known that any irreducible $\gP$-module arises in this way. 

Let us recall that the homogeneous space $\G/\gP$ of the Cartan geometry of type $(\G,\gP)$ carries a natural distribution of co-dimension ${k\choose2}$. Let $p:\G\ra\G/\gP,\ g\mapsto g.\gP$ be the canonical projection and  $\ell_g:\G\ra\G$ and $r^g:\G\ra\G$ be the multiplication by $g\in\G$ on the left and on the right, respectively. As $\ell_g$ is a diffeomorphism, the map $T\ell_{g^{-1}}:T_g\G\ra T_e\G=\lag$ is a linear isomorphism and the 1-form $\omega\in\Omega^1(\G,\lag)$  defined by $\omega_g(X)= T\ell_{g^{-1}}(X),\ X\in T_g\G$ is the \textit{Maurer-Cartan form} on $\G$. We put $T^{-1}_g\G:=\omega^{-1}_g(\lag^{-1})$ so that $T^{-1}\G:=\bigcup_{g\in\G}T^{-1}_g\G$ is a distribution on $\G$. As $\omega$ is $\gP$-equivariant in the sense that for every $g\in\G:\ (r^g)^\ast\omega=\Ad(g^{-1})\circ\omega$ and $\lag^{-1}$ is a $\gP$-invariant subspace, it follows that $T^{-1}\G$ projects under $Tp$ to a well-defined distribution $H$ of co-dimension ${k\choose2}$.

\medskip

Assume now that $\mV$ is an irreducible $\gP$-module. Then the space $\Gamma(V)$ of smooth sections of the associated vector bundle $V:=\G\times_\gP\mV$ is by \cite[Proposition 1.2.7]{CS} isomorphic to the space $\cC^\infty(\G,\mV)^\gP$ of smooth $\mV$-valued $\gP$-equivariant functions on $\G$.  More generally, if $\cU$ is an open subset of the homogeneous space and $V|_\cU$ is the restriction of $V$ to $\cU$, then $\Gamma(V|_\cU)\cong\cC^\infty(p^{-1}(\cU),\mV)^\gP$. The group $\G$ has a canonical action on $\cC^\infty(\G,\mV)$ given by $ g.f= f\circ\ell_{g^{-1}}$ where $f\in\cC^\infty(\G,\mV)$ is a smooth $\mV$-valued function on $\G$. As the multiplication on the right commutes with the multiplication on the left, this action descends to an action on $\cC^\infty(\G,\mV)^\gP$. Using the isomorphism above, this induces a $\G$-action on $\Gamma(V)$.

Assume that $\mW$ is another  irreducible $\gP$-module and that 
\begin{equation}\label{global invariant operator}
D:\Gamma(V)\ra\Gamma(W) 
\end{equation}
is a linear differential operator of order $r$ where $W:=\G\times_\gP\mW$. We call $D$ a $\G$-\textit{invariant} operator (or simply invariant) if it intertwines  the action of $\G$ on $\Gamma(V)$ and $\Gamma(W)$. Let $\cJ^r\mV$ be the vector space of $r$-jets of germs of sections of  $V$ at the origin $e\gP\in\G/\gP$ where $e\in\G$ is the identity element. Then $\cJ^r\mV$ is a $\gP$-module and $D$  determines a $\gP$-equivariant  map
\begin{equation}\label{P-map over origin}
\cJ^r\mV\ra\mW, \ \ \ j^rf\mapsto Df(e\gP)
\end{equation}
where $j^rf$ is the $r$-th jet of a germ $f$ of a smooth section of $V$ at $e\gP$.
It is well known (see \cite[Section 1.4.10]{CS}) that this assignment is a bijection, i.e. a $\gP$-equivariant homomorphism $\cJ^r\mV\ra\mW$ determines a unique linear and $\G$-invariant differential operator of order at most $r$. The order of the induced operator is $<r$ if the homomorphism (\ref{P-map over origin}) factorizes through the canonical projection $\cJ^r\mV\ra\cJ^{r-1}\mV$.

\subsection{Affine subset of the homogeneous space}\label{section affine subset}
The group $\G$ has a canonical action on the Grassmannian variety  of totally isotropic  $k$-dimensional subspaces in $\R^{k,2n+k}$. It is straightforward to verify that $\gP$ is the stabilizer of the totally isotropic subspace $[e_1,\dots,e_k]$. Hence, we can view $\G/\gP$  as the isotropic Grassmannian and we will do that without further comment. We will now consider an affine subset of the homogeneous space. We will use the following convention. If $A=(a_{ij})\in M(2m,k,\R)$ has maximal rank $k$, then we denote by $[a_{ij}]$ the $k$-dimensional subspace that is spanned by the columns of the matrix $A$. 
\medskip

The restriction of  $p\circ\exp$ to $\lag_-$ induces a map
\begin{align}\label{coordinates on osu}
\lag_-&\ra\G/\gP, \ \ 
\left(
\begin{array}{ccc}
0&0&0\\
X&0&0\\
Y&-X^T&0\\
\end{array}
\right) 
\mapsto
\left[
\begin{matrix}
1_k\\
X\\
Y-\frac{1}{2}X^TX\\
\end{matrix}
\right].
\end{align}
The map is injective and a moment of thought shows that it is a diffeomorphism onto its image which  is an open, dense and affine subset of $\G/\gP$ which is known as the "big cell", see \cite[Example 5.1.12]{CS}. Let us  write $X=(x_{\alpha i}),\ Y=(y_{rs})$ so that we can use $x_{\alpha i},\ \alpha=1,\dots,2n,\ i=1,\dots,k$ and $y_{rs},\ 1\le r<s\le k$ as coordinates on the big cell. It will be convenient to view the big cell as a Lie group with Lie algebra $\lag_-$. 

Put $\G_-:=\exp(\lag_-)$. Then $\G_-$ is a closed subgroup of $\G$ with Lie algebra $\lag_-$, the exponential map is a diffeomorphism $\lag_-\ra\G_-$ and  thus, the composition $\G_-\xrightarrow{\iota}\G\xrightarrow{p}\G/\gP$, where $\iota$ is the inclusion, induces a diffeomorphism between $\G_-$ and the big cell.  Hence, we can view $\G_-$ as the big cell
and the inclusion $\iota$ as a section of $p^{-1}(\G_-)\ra\G_-$. This section determines a trivialization  $\Phi:\G_-\times\gP\ra p^{-1}(\G_-)$ of the $\gP$-principal bundle over $\G_-$. As $p^{-1}(\G_-)\subset\G$, the map $\Phi$ is simply given by  multiplication on $\G$.

\medskip

Let  $\mV$ be a $\gP$-module and $V$ be the associated vector bundle as in Section \ref{section homogeneous space}. Recall that  $\Gamma(V|_{\G_-})\cong\cC^\infty(p^{-1}(\G_-),\mV)^\gP$. As any $\gP$-equivariant function on $p^{-1}(\G_-)$ is determined by its restriction to $\G_-$, it follows that the map  
$$\cC^\infty(p^{-1}(\G_-),\mV)^\gP\ra\cC^\infty(\G_-,\mV),\ f\mapsto f\circ\iota$$
is an isomorphism of vector spaces. We see that the restriction of the differential operator (\ref{global invariant operator}) to $\G_-$ induces (in the trivialization $\iota$) a linear differential operator 
\begin{equation}\label{invariant operator over affine subset}
 \cC^\infty(\G_-,\mV)\ra\cC^\infty(\G_-,\mW).
\end{equation}
We will denote this operator also by $D$ as there is no risk of confusion. If the operator $D$ is invariant, then it is straightforward to verify that  (\ref{invariant operator over affine subset}) is a $\G_-$-invariant differential operator with respect to the canonical $\G_-$-action on both spaces (that is induced by multiplication on left). Let us now recall well known classification of linear and $\G_-$-invariant differential operators.

\subsection{Invariant differential operators and weighted jets}\label{section invariant op and wjets}
Let us first recall the definition of the universal enveloping algebra associated to $\lag_-$.  In order to simplify notation, we write $gr(X)=i$ if, and only if $X\in\lag_i$ and $i\in\{-1,-2\}$. Recall Section \ref{section lie algebra} that  the Lie bracket on $\lag_-$ is compatible with the grading in the sense that  $[\lag_{i},\lag_{j}]\subset\lag_{i+j}$ where $i,j=-1,-2$ and we agree that $\lag_{\ell}=\{0\}$ whenever $\ell<-2$. 

\medskip 

The \textit{universal enveloping algebra} $\cU(\lag_-)$ associated to $\lag_-$ is the quotient of the tensor algebra $T(\lag_-)$ of $\lag_-$ by the both sided ideal $I$  that  is generated by the elements of the form: $X\otimes Y-Y\otimes X-[X,Y]$, \ $X,Y\in\lag_-$. The grading on $\lag_-$ induces a grading  $T(\lag_-)=\bigoplus_{r\ge0}T_r(\lag_-)$ on the tensor algebra  where $T_{r}(\lag_-)$ is the linear span of all elements  $X_{i_1}\otimes\dots \otimes X_{i_\ell}$ where $\sum_{j=1}^\ell gr(X_{i_j})=-r$. Notice that the  grading is compatible with  multiplication.  As the Lie bracket on $\lag_-$ is compatible with the grading, it is easy to see that $I=\bigoplus_{r\ge0}(T_r(\lag_-)\cap I)$. Thus the grading on the tensor algebra induces a grading  $\cU(\lag_-)=\bigoplus_{r\ge0}\ \cU_{r}(\lag_-)$  where $\cU_{r}(\lag_-):=T_{r}(\lag_-)/(T_{r}(\lag_-)\cap I)$. This grading will be crucial later on.

It is well known that $\cU(\lag_-)$ is isomorphic to the algebra of left invariant differential operators acting on $\cC^\infty(\G_-)$. On the generators of $\cU(\lag_-)$, the isomorphism is given by $X+I\mapsto L_X$ where $X\in\lag_-$ and $L_X$ is the left invariant vector on $\G_-$ whose value at the identity element $e$  is $X$. In the sequel, we will view the algebra of linear differential operators acting $\cC^\infty(\G_-)$  as $\cU(\lag_-)$ without further comment.

A straightforward computation shows that
\begin{equation}\label{invariant vector fields in coordinates}
L_{e^i\otimes\varepsilon_\alpha}=\partial_{x_{\alpha i}}-\sum_{j=1}^k\frac{1}{2}x_{\alpha j}\partial_{y_{ij}}\ \ \mathrm{and}\ \ L_{e^r\wedge e^s}=\partial_{y_{rs}} 
\end{equation}
where we use the bases of $\lag_{-1}$ and $\lag_{-2}$ given in (\ref{bases of g-})  and the convention $\partial_{y_{rs}}=-\partial_{y_{sr}}$. Check that the  bracket of vector fields is compatible with the Lie bracket on $\lag_-$ given in (\ref{lie bracket on g-}). 
From the definition, it immediately follows that $H|_{\G_-}$ is (as a $\cC^\infty(\G_-)$-module) generated by $\{L_X:\ X\in\lag_{-1}\}$. It follows
from (\ref{invariant vector fields in coordinates})  that any $D\in\cU(\lag_-)$ is a differential operator with polynomial coefficients. More generally, the vector space of linear and left invariant differential operators $\cC^\infty(\G_-,\mV)\ra\cC^\infty(\G_-,\mW)$ is isomorphic to the tensor product $L(\mV,\mW)\otimes\cU(\lag_-)$.

\medskip

The algebra of polynomials on the vector  space $\lag_-$  is naturally isomorphic to the symmetric algebra $S(\lap_+)$. The grading $\lap_+=\lag_1\oplus\lag_2$ induces a grading 
\begin{equation}\label{polynomial over affine set}
S(\lap_+)=\bigoplus_{r\ge0}\xS^r\lap_+\ \ \mathrm{where} \ \ \xS^r\lap_+=\bigoplus_{\ell=0}^{\lfloor \frac{r}{2}{\rfloor}} S^\ell\lag_2\otimes S^{r-2\ell}\lag_1 
\end{equation}
and $\lfloor\ \rfloor$ denotes the integer part. The grading on $S(\lap_+)$ is obviously compatible with  multiplication and so $S(\lap_+)$ is a graded algebra. As $\lag_-$ is diffeomorphic to the big cell $\G_-$, we can view $S(\lap_+)$ as polynomials on $\G_-$. Then there is a natural pairing 
\begin{equation}\label{duality}
 \cU(\lag_-)\otimes S(\lap_+)\ra\R,\ \ (D,f)\mapsto Df(e)
\end{equation}
which yields a duality $\cU_r(\lag_-)^\ast\cong \xS^r\lap_+$.

\medskip

Finally, let us recall the concept of weighted jets. Let $\cE(\mV)$ be the vector space of germs of smooth sections of the associated bundle $V$ at $e\gP$ from Section \ref{section homogeneous space}. If  $f\in\cE(\mV)$, then $\underline f:=f\circ\iota$ is a germ of a smooth $\mV$-valued function at $e\in\G_-$.  We write $f\sim_r g$ if $g\in\cE(\mV)$ and $D\underline f(e)=D\underline g(e)$ for each $D\in\cU_\ell(\lag_-),\ \ell\le r$. It is well known that the equivalence class $\sim_r$ of $f$ does not depend on the choice of the trivialization. We denote by $\xJ^r\mV$ the quotient of $\cE(\mV)$ by the subspace $\{f\in\cE(\mV):\ f\sim_{r+1}0\}$ and by $\xj^r f$ the class of $f$ in $\xJ^r\mV$. We call $\xJ^r\mV$  the space of \textit{weighted $r$-jets at $e$ associated to} $\mV$ and $\xj^rf$  the \textit{weighted $r$-jet} of $f$ at $e$.

As the subspace $\{f\in\cE(\mV):\ f\sim_{r+1}0\}$ is $\gP$-invariant, it follows that the standard action of $\gP$ on $\cE(\mV)$ descends to an action on $\xJ^r\mV$. As  $f\sim_r0$ whenever $f\sim_{r+1}0$, it follows that there is a canonical  surjective map $\xJ^{r}\mV\ra\xJ^{r-1}\mV$ and we denote its  kernel  by $\xgr^r\mV$. It is well known (see \cite{Mo}, \cite{MoI}, \cite{Ne} or \cite{TSII}) that there are linear isomorphisms
\begin{equation}\label{isom weighted jets}
\xgr^r\mV\ \cong\ \cU_r(\lag_-)^\ast\otimes\mV\ \cong\ \xS^r\lap_+\otimes\mV. 
\end{equation}
From the definitions, it easily follows  that for each $r\ge0$ there is a well-defined linear map $\cJ^r\mV\ra\xJ^r\mV$ which is surjective.  Here $\cJ^r\mV$ denotes the vector space of ordinary (unweighted) $r$-jets.

\begin{prop}\label{thm dif op through jets}
Suppose that $\mV$ and $\mW$ are irreducible $\G_0$-modules on which the grading element $E$  acts with eigenvalues $\lambda$ and $\mu$, respectively and that there is a nonzero, linear and $\G$-invariant differential operator (\ref{global invariant operator}). 

Then $r:=\lambda-\mu$ is a non-negative integer and:
\begin{enumerate}[(I)]
\item the induced  $\gP$-equivariant homomorphism (\ref{P-map over origin}) descends to a $\gP$-equivariant map
\begin{equation}\label{weighted P-equivariant map over origin}
\phi:\xJ^r\mV\ra\mW
\end{equation} 
which does not factorize through the canonical projection  $\xJ^r\mV\ra\xJ^{r-1}\mV$ and
\item the invariant differential operator (\ref{invariant operator over affine subset}) induced by $D$ belongs to $\cU_r(\lag_-)$.
\end{enumerate}
We call the integer $r$ the \textbf{weighted order} of the differential operator $D$.
 \end{prop}
 
\begin{proof}
By  Section \ref{section homogeneous space}, the differential operator $D$ is determined by a $\gP$-equivariant homomorphism $\lambda:\cJ^\ell\mV\ra\mW$ for some integer $\ell$. As we consider a $|2|$-graded geometry, there is a well-defined surjective map $\xJ^{2\ell}\mV\ra\cJ^\ell\mV$ which is $\gP$-equivariant and hence, there is an integer $r\ge0$ such that the operator $D$  induces and is completely determined by a $\gP$-equivariant map as in (\ref{weighted P-equivariant map over origin}). The map is invariantly defined by $\phi(\xj^rf)=Df(e\gP)$ where $f\in\cE(\mV)$. Without loss of generality we may assume that $\phi$ does not factorize through the canonical projection $\xJ^r\mV\ra\xJ^{r-1}\mV$ which is equivalent to the fact that $\phi$ is non-zero on $\xgr^r\mV$. Hence, $\phi$ induces an isomorphism between $\mW$ and an irreducible $\gP$-invariant subspace of $\xgr^r\mV$. By the assumptions, the grading element acts diagonalizably on $\xJ^r\mV$ with finitely many eigenvalues which are of the form $\lambda,\lambda+1,\lambda+2,\dots,\lambda+r$ where $\xgr^r\mV$ is the  eigenspace of $\lambda+r$. Recall the linear isomorphism $\xgr^r\mV\cong\xS^r\otimes\mV$ from (\ref{isom weighted jets}). From this  we see that $r$ is a non-negative integer and  also the claim in (I). 

From the definition of $\phi$ follows that the value of $D f$ at the origin $e\gP$ depends only on $\xj^rf$. This implies that  $D$, viewed as the $\G_-$-invariant operator (\ref{invariant operator over affine subset}), belongs to $\bigoplus_{i=0}^r\cU_{i}(\lag_-)$. Since $Df(e)=0$ whenever $\xj^{r-1}_ef=0$, it follows that $D\in\cU_r(\lag_-)$. The proof is complete.
\end{proof}

As we have seen in the proof of Proposition \ref{thm dif op through jets}, if $D$ is determined by the map (\ref{weighted P-equivariant map over origin}), then $Df(e\gP)$ where $f\in\cE(\mV)$ depends only $\xj^rf$. More generally, one can show (see \cite{Mo} or \cite{Ne}) that $\xj^sDf$ depends only $\xj^{r+s}f$ and thus, the operator $D$ induces for each $s\ge0$ a linear map
\begin{equation}\label{prolongation of differential operator}
\xgr^{r+s}D:\xgr^{r+s}\mV\ra\xgr^s\mW,\ \ \xj^{r+s}f\mapsto\xj^sDf
\end{equation}
which is called the $s$-\textit{th prolongation of} $D$. From the definition of the map (\ref{prolongation of differential operator}), it follows that there is a commutative diagram
\begin{equation}\label{commutative diagram with polynomials and weighted jets}
 \xymatrix{\xS^{r+s}\lap_+\otimes\mV\ar[d]\ar[r]^{D}&\ar[d]\xS^s\lap_+\otimes\mW\\
\xgr^{r+s}\mV\ar[r]^{\xgr^{r+s}D}&\xgr^s\mW &}
\end{equation}
where the vertical arrows are the isomorphisms from (\ref{isom weighted jets}) and in the first row we view  $\xS^{r+s}\lap_+\otimes\mV$ and $\xS^s\lap_+\otimes\mW$ as the vector spaces of $\mV$ and $\mW$-valued polynomials on $\G_-$, respectively, and $D$ as the corresponding differential operator on $\G_-$.

\subsection{Descending procedure}
\label{section descending natural operators}

Let $\G_{-2}:=\exp(\lag_{-2})$. Then it is easy to see that $\G_{-2}$ is a closed, abelian and normal subgroup of $\G_-$  with Lie algebra $\lag_{-2}$ and that $\exp:\lag_{-2}\ra\G_{-2}$ is a diffeomorphism. We will be interested in the space $\G_{-2}\setminus \G_-$ of right cosets which we for brevity denote by $U$. It is straightforward  to verify that the composition 
$$M(2n,k,\R)\cong\lag_{-1}\xrightarrow{\exp}\G_-\xrightarrow{q}\G_{-2}/\G_-= U,$$ 
where  $q$ is the canonical projection, is a diffeomorphism. For brevity,
we denote   the origin $\G_{-2}e$ of $U$ by $x_0$. 

Let $J_{x_0}^r(U,\mV)$ be  the vector space of usual $r$-jets of germs of smooth $\mV$-valued functions at the origin $x_0$. We denote by $gr^r\mV$ the kernel of the canonical projection $J^{r}_{x_0}(U,\mV)\ra J^{r-1}_{x_0}(U,\mV)$. The projection  $q$ induces a linear map  
\begin{equation}
q^\ast:\cC^\infty(U,\mV)\ra\cC^\infty(\G_-,\mV),\ \ f\mapsto f\circ q
\end{equation}
which is obviously injective. This map descends for each $r\ge0$ to injective linear maps
\begin{equation}\label{inclusion of jets}
J_{x_0}^r(U,\mV)\hookrightarrow\xJ^r\mV\ \ \mathrm{and} \ \ gr^r\mV\hookrightarrow\xgr^r\mV
\end{equation}
which we also denote by $q^\ast$. Notice that in this way, $gr^r\mV$ gets identified with the subspace $S^r\lag_1\otimes\mV\subset\xS^r\otimes\mV\cong\xgr^r\mV$.

\begin{prop}\label{thm descending}
Suppose that $\mV$ and $\mW$ are vector spaces of dimensions $s$ and $t$, respectively. Assume that 
\begin{equation}
{D}:\cC^\infty(\G_-,\mV)\ra\cC^\infty(\G_-,\mW). 
\end{equation}
is a linear and  $\G_-$-invariant operator differential operator from $\cU_r(\lag_-)$.
Then there is a unique linear differential operator 
\begin{equation}\label{descended operator}
 \underline D:\cC^\infty(U,\mV)\ra\cC^\infty(U,\mW)
\end{equation}
such that the following diagram
\begin{equation}
\xymatrix{\cC^\infty(\G_-,\mV)\ar[r]^{{D}}&\cC^\infty(\G_-,\mW)\\
\cC^\infty(U,\mV)\ar[u]^{q^\ast}\ar[r]^{\underline D}&\cC^\infty(U,\mW)\ar[u]^{q^\ast}} 
\end{equation}
commutes.
Moreover, $\underline D$ is a homogeneous constant coefficient operator of order $r$.
\end{prop}

\begin{proof}
By the definition, the invariance of ${D}$ means that  $({D} f)\circ\ell_{g^{-1}}={D}(f\circ\ell_{g^{-1}})$ for every $g\in\G_-$  where $\ell_{g^{-1}}:\G_-\ra\G_-$ is  multiplication by $g^{-1}$ on left.
Assume now that  $f\in\cC^\infty(\G_-,\mV)$ belongs to the image of $q^\ast$. It is elementary to see that this is equivalent to  $f\circ\ell_{g^{-1}}=f$ for every $g\in\G_{-2}$. Thus we see that there is a linear operator (\ref{descended operator}) as claimed. It remains to show that $\underline D$ is a homogeneous differential operator of order at most $r$ (which may be possibly zero). From (\ref{invariant vector fields in coordinates}), it easily follows that  for each $h\in\cC^\infty(U,\mV): \ L_{e^i\otimes\varepsilon_\alpha}(q^\ast h)=q^{\ast}\partial_{x_{\alpha i}}h$ and that $L_{e^r\wedge e^s}(q^\ast h)=0$. Hence, the last claim follows from the definition of $\cU_r(\lag_-)$ given in Section \ref{section invariant op and wjets}.
\end{proof}

With the notation from Proposition \ref{thm descending}, notice that there is for each $s\ge0$  a commutative diagram of linear maps
\begin{equation}\label{commutative diagram with jets and weighted jets}
\xymatrix{\xgr^{r+s}\mV\ar[rr]^{\xgr^{r+s}{D}}&&\xgr^s\mW\\
gr^{r+s}\mV\ar[u]^{q^\ast}\ar[rr]^{gr^{r+s}\underline  D}&&gr^s\mW\ar[u]^{q^\ast}} 
\end{equation}
where $gr^{r+s}\underline D$ is the standard linear map $j^{r+s} f\mapsto j^s \underline Df$ that is associated to any linear differential operator of order $r$ and $j^rf$ is the usual $r$-jet of $f$ at $x_0$.

\subsection{Complex parabolic geometry}\label{section embedding}
We will now consider the complex analogue of the parabolic geometry of type $(\G,\gP)$ from  Sections \ref{section lie algebra} and \ref{section homogeneous space}.  We will proceed in this Section  rather quickly as most  constructions are similar to the real case and we refer to \cite{S} for a more thorough introduction if necessary.

\medskip

We have $\C^{2m}=\R^{2m}\otimes\C$ and we denote by  $h^\C$ the $\C$-bilinear extension of $h$ to $\C^{2m}$. The complexification of $\lag$ is  $\laso(2m,\C)\cong\lag^\C:=\lag\otimes_\R\C$ and the direct sum decomposition $\lag^\C=\lag_{-2}^\C\oplus\lag_{-1}^\C\oplus\dots\oplus\lag_2^\C$, where $\lag_i^\C:=\lag_i\otimes_\R\C$, is  a $|2|$-grading on $\lag^\C$. The group $\G^\C:=\Spin(2m,\C)$ is a connected and simply connected complex Lie group with Lie algebra $\lag^\C$. There is an embedding $\zeta:\G\ra\G^\C$ such that $T_e\zeta:T_e\G\ra T_e\G^\C$ is the canonical inclusion $\lag=\lag\otimes\R\hookrightarrow\lag\otimes\C=\lag^\C$.
We denote by  $\gP^\C$ the parabolic subgroup associated  via (\ref{parabolic subgroup})  to the $|2|$-grading and  by $\G_0^\C$ the Levi factor which is defined as in (\ref{Levi factor}). The group $\G_0^\C$ is isomorphic to $\GL(k,\C)\times\Spin(2n,\C)$. It is obvious that $\zeta$ restricts to embeddings  $\gP\hookrightarrow\gP^\C$ and $\G_0\hookrightarrow\G_0^\C$ which also shows that $\zeta$ descends to an embedding $\varsigma:\G/\gP\hookrightarrow\G^\C/\gP^\C$.

The exponential image $\G_-^\C$ of $\lag_-^\C:=\lag_-\otimes\C$ is a closed subgroup of $\G^\C$ with  embedding $\iota^\C:\G_-^\C\ra\G^\C$. The composition $\G_-^\C \xrightarrow{\iota^\C}\G^\C\xrightarrow{p^\C}\G^\C/\gP^\C$ is a biholomorphism onto its image which identifies $\G_-^\C$ with an open, dense and affine subset of $\G^\C/\gP^\C$. It is clear that $\zeta$ restricts to an embedding $\G_-\hookrightarrow\G_-^\C$ and that the pair $\G_-\subset\G_-^\C$ looks as the pair $\R^\ell\subset\C^\ell$ where we for a moment put $\ell:={2nk+{k\choose2}}$. Hence,  $\G_-^\C$ is a complexification of the real analytic manifold $\G_-$. Since $\G/\gP$ and $\G^\C/\gP^\C$ are homogeneous spaces, also $\G^\C/\gP^\C$ is a complexification of the real analytic manifold $\G/\gP$ as in \cite[Definition 6.6]{A}.

\medskip

Assume now that $\mV$ is a complex irreducible $\gP^\C$-module. Then $\mV$ is also a complex irreducible  $\gP$-module by restriction. Let $\cO(\mV)$ be the space of germs of holomorphic sections of $V^\C:=\G^\C\times_{\gP^\C}\mV$ at  $e\gP^\C$ and $\cA^\omega(\mV)$ 
be the space of germs of real analytic sections of $V:=\G\times_\gP\mV$ at $e\gP$. If we view sections as equivariant functions, then it is easy to see that the pullback of $f\in\cO(\mV)$ along the inclusion $\zeta$ is  an element of $\cA^\omega(\mV)$. Hence, there is a well-defined map
\begin{equation}\label{isom hol and analytic germs}
 \zeta^\ast:\cO(\mV)\ra\cA^\omega(\mV).
\end{equation}
As $\G^\C/\gP^\C$ is a complexification of the real analytic manifold $\G/\gP$, the  following fact is well known.

\begin{lemma}
 The map (\ref{isom hol and analytic germs}) is an isomorphism of infinite dimensional vector spaces.
\end{lemma}

Recall Section \ref{section invariant op and wjets} for the definition of $\xJ^r\mV, \ r\ge0$ and that there is a canonical projection $\cA^\omega(\mV)\ra\xJ^r\mV$. This  induces via the isomorphism (\ref{isom hol and analytic germs}) a surjective linear map $\cO(\mV)\ra\xJ^r\mV$. Hence, we can view $\xJ^r\mV$ as a quotient of $\cO(\mV)$ and $\xgr^r\mV$ as a subspace of this quotient space. Also we can  define the $r$-th weighted jet for any germ $f\in\cO(\mV)$. This shows that the vector spaces $\xJ^r\mV$ and $\xgr^r\mV$ are naturally $\gP^\C$-modules.

Assume now that $\mW$ is a complex irreducible representation of $\gP^\C$ and that there is a $\gP$-equivariant homomorphism (\ref{weighted P-equivariant map over origin}). Then the homomorphism is also $\gP^\C$-equivariant and it induces a $\G^\C$-invariant differential operator just as in the real case. Conversely, any such linear and $\G^\C$-invariant differential operator induces a $\gP^\C$-equivariant homomorphism as in (\ref{weighted P-equivariant map over origin}).  As this observation will be needed later on, let us formulate it as a lemma.

\begin{lemma}
\label{lemma real and holomorphic invariant operators}
Assume that $\mV$ and $\mW$ are complex and irreducible  representations of $\gP^\C$ and that there is a $\gP^\C$-equivariant homomorphism as in (\ref{weighted P-equivariant map over origin}). Then the homomorphism induces a linear and $\G^\C$-invariant operator 
\begin{equation}\label{holomorphic operator}
\G^\C\times_{\gP^\C}\mV\ra\G^\C\times_{\gP^\C}\mW
\end{equation}
of order at most $r$ where we for brevity write bundles instead of the associated sheaves of germs of holomorphic sections. 

Conversely, a linear and $\G^\C$-invariant differential operator  (\ref{holomorphic operator}) induces a $\gP^\C$-equivariant map (\ref{weighted P-equivariant map over origin}) for some integer $r\ge0$.
\end{lemma}

\section{$k$-Dirac complex and the resolution of the $k$-Dirac operator}\label{Section k-Dirac complex}
In the current Section we will apply the theory that we built in  Section 2. In Section \ref{section YD} we will  set some notation. We will show (see Theorem \ref{thm k-Dirac complex})  that there is a complex (the $k$-Dirac complex from Introduction) of linear and $\G$-invariant differential operators on $\G/\gP$  and that this complex is formally exact (see Proposition \ref{thm formal exactness}).  This uses the results already established in \cite{S} and the discussion from Section \ref{section embedding}.  The $k$-Dirac complex induces a complex of linear and  invariant differential operators with polynomial coefficients on $\G_-$ and we will show in Section \ref{section exactness of k-Dirac complex with formal power series} that this complex is exact in the sense of formal power series at any given point (see Theorem \ref{thm exactnes with formal power series}). By Proposition \ref{thm descending}, the complex on $\G_-$ induces a complex of linear and constant coefficient differential operators on the set $U:=\G_{-2}\setminus\G_-$. We will prove in Section \ref{section resolution} that the first operator in the descended complex is the $k$-Dirac operator from Introduction and that the complex is locally exact (see Theorem \ref{the main theorem}).

\subsection{Young diagrams}\label{section YD}
We will use the following notation.
We put $$\N^{k,n}_{++}:=\{(a_1,\dots,a_k): a_i\in\Z,\ n\ge a_1\ge a_2\ge\dots\ge a_k\ge0\}$$ 
and  call $a\in\N^{k,n}_{++}$ a  \textit{partition} of the number $|a|:=a_1+\dots+a_k$.  
To the partition $a$ we associate a diagram that consists of left justified $a_i$ boxes in the $i$-th row and call it the \textit{Young diagram  associated to} $a$.  We denote by $S^k$ the subset of $\N^{k,n}_{++}$ consisting of those partitions whose Young diagram is symmetric with respect to the main diagonal. Then we define $q(a)$ and $d(a)$ as the number of boxes in the Young diagram that are "above" and "on" the main diagonal, respectively. We put $r(a):=d(a)+q(a)$ and $S^k_j:=\{a\in S^k:\ r(a)=j\}$.
If $a'=(a'_1,\dots,a'_k)\in\N^{k,n}_{++}$, then we  write $a< a'$ if $a\ne a'$  and  $a_i\le a_i'$ for each $i=1,\dots,k$. 

\medskip

\begin{ex}
\begin{enumerate}
 \item The  Young diagram associated to $a=(4,2,0)\in\N^{3,4}_{++}$ is $\Ycd$. Then $d(a)=2,\ q(a)=3$ and $a\not\in S^k$ as there are four boxes in the first row but only two boxes in the first column.
 \item The Young diagram associated to $a'=(3,2,1,0)\in\N^{4,6}_{++}$  is $\Ytdj$. Hence, $d(a)=q(a)=2$ and we see that $a'\in S^k$.
\end{enumerate}
\end{ex}

\begin{ex}
\begin{enumerate}
\item If $k=2$, then $$S^{2}=\{\cdot\ <\ \Yj\ <\ \Ydj\ <\ \Ydd\big\}$$ and
$$S^2_0:=\big\{\cdot\big\},\ S^2_1:=\big\{\Yj\big\},\ S^2_2:=\big\{\Ydj\big\},\ S^2_3:=\big\{\Ydd\big\}$$
where we write also the ordering on $S^2$ defined above.
\item   For $k=3$ we have
$$S^{3}=
\left\{
\begin{array}{ccccccccccc}
\cdot&<&\Yj&<&\Ydj&<&\Ydd&&&\\ 
&&&&\wedge&&\wedge&&&&\\
&&&&\Ytjj&<&\Ytdj&<&\Yttd&<&\Yttt\\
\end{array}
\right\}$$
and 
\begin{align*}
& S^3_0:=\big\{\cdot\big\},\ S^3_1:=\big\{\Yj\big\},\ S^3_2:=\big\{\Ydj\big\},\ S^3_3:=\big\{\Ydd,\ \Ytjj\big\},\\
 &S^3_4:=\big\{\Ytdj\big\},\ S^3_5:=\big\{\Yttd\big\},\ S^3_6:=\big\{\Yttt\big\}.
\end{align*}
\end{enumerate}
\end{ex}

See \cite{F} that there is a recursive pattern that relates $S^{k+1}$ to $S^{k}$.

\subsection{The k-Dirac complex}\label{section k-Dirac complex}

For $a=(a_1,\dots,a_k)\in S^k$ we denote by  $\mW_a$  an irreducible $\GL(k,\C)$-module with lowest  weight 
$$-\bigg(\frac{1-2n}{2}-a_k,\dots,\frac{1-2n}{2}-a_1\bigg).$$ 
Let $\Sp$ be the complex space of spinors as in Introduction. Recall that $\Sp=\Sp_+\oplus\Sp_-$ where $\Sp_+,\ \Sp_-$  are two non-isomorphic complex spinor modules  of $\Spin(2n,\C)$ as in \cite{S}.
Then\footnote{Notice that we denoted this module in \cite{S} by $\mV_{\mu_a}$. More precisely $\mV_{\mu_a}$ is a summand of $\mV_a$ given by choosing $\Sp_+$ or $\Sp_-$ in $\Sp$.} $\mV_a:=\mW_a\ \!\boxtimes\ \!\Sp$ is an irreducible complex $\G_0^\C$-module and thus  also an irreducible  complex representation of $\gP^\C$ and $\gP$. The grading element $E$ acts  on $\mV_a$ by multiplication by  $\frac{k(n-1)}{2}+|a|$. 
\medskip

We put
\begin{equation}
\mV_j:=\bigoplus_{a\in S^k_j}\mV_{a},\ \ \  V_a:=\G\times_{\gP}\mV_a \ \ \ \mathrm{and} \ \ \ V_j=\G\times_{\gP}\mV_j.
\end{equation}
For $a\in S^k_j$, we denote by $s_a$ the $a$-th component of  $s\in\Gamma(\mV_j)$ so that  $s=(s_a)_{a\in S^k_j}$. 

\begin{prop}\label{thm k-Dirac complex}
Assume that $n\ge k\ge2,\ j\ge0,\ a\in S^k_j,\ a'\in S^k_{j+1}$ and that $a<a'$. 
Put $r:=|a'|-|a|$. 
\begin{enumerate}[(a)]
 \item Then  $1\le r \le2$ and there is  an invariant and $\C$-linear differential operator 
\begin{equation}\label{single dif op in k-Dirac complex}
{D}^a_{a'}:\Gamma(V_a)\ra\Gamma(V_{a'})
\end{equation}
of the weighted order  $r$.
\item 
There is a complex of differential operators
\begin{equation}\label{k-Dirac complex}
\Gamma(V_0)\xrightarrow{{D}_0}\Gamma(V_1)\xrightarrow{{D}_1}\dots\xrightarrow{{D}_{j-1}}\Gamma(V_j)\xrightarrow{{D}_j}\Gamma(V_{j+1})\xrightarrow{{D}_{j+1}}\dots
\end{equation} 
where 
\begin{equation}\label{dif op in k-Dirac complex}
 ({D}_j s)_{a'}:=\!\!\!\!\!\sum_{a\in S^k_j,a<a'}\!\!\!\!\! {D}_{a'}^{a}s_{a}, \ a'\in S^k_{j+1},\ s\in\Gamma(V_j).
\end{equation}
\end{enumerate}
\end{prop}

\begin{proof}
(a) The bounds on $r$ follow easily from the definitions. If there is a nonzero, linear and $\G$-invariant differential operator (\ref{single dif op in k-Dirac complex}), then from the action of the grading element on $\mV_a$ and $\mV_{a'}$ given above and Proposition \ref{thm dif op through jets}, it  follows that it is a differential operator of the weighted order $r$. Hence, it is enough to show that there is a $\gP$-equivariant map $\xJ^\ell\mV_a\ra\mV_{a'}$ for some integer $\ell\ge0$. But this follows from  Lemmata \ref{lemma real and holomorphic invariant operators},  \cite[6.1 and 7.8]{S}.

(b) This easily follows from \cite[Theorem 6.2]{S} and the isomorphism (\ref{isom hol and analytic germs}).
\end{proof}

We call the complex of invariant differential operators from  Proposition \ref{thm k-Dirac complex} the $k$-\textit{Dirac complex}.
Let $a,a',r$ and $D_{a'}^a$ be as in Proposition \ref{thm k-Dirac complex}. Recall (\ref{prolongation of differential operator}) that the $s$-th prolongation of $D_{a'}^a$ is a linear map
\begin{equation}\label{linear map on weighted jets}
\xgr D_{a'}^a: \xgr^{r+s}\mV_{a'}\ra\xgr^s\mV_a.
\end{equation}
This map is homogeneous of degree $-r$. It will be convenient to shift the gradings  so that the map (\ref{linear map on weighted jets}) is homogeneous of degree $-1$. For this we  notice that $q(a')=q(a)+r-1$ and thus, if we put
\begin{equation}\label{shifted grading}
 \xgr^i\mV_a[\uparrow]:=\xgr^{i-q(a)}\mV_a
\end{equation}
the map (\ref{linear map on weighted jets}) becomes
 \begin{equation}\label{linear map on shifted weighted jets}
\xgr D^a_{a'}:\xgr^i\mV_a[\uparrow]\ra\xgr^{i-1}\mV_{a'}[\uparrow]
\end{equation}
where $i=r+s+q(a)$. If we put
\begin{equation}
 \xgr^i\mV_j[\uparrow]=\bigoplus_{a\in S^k_j}\xgr^i\mV_{a}[\uparrow],
\end{equation}
then the differential operator $D_j$ induces for each $i\ge0$ a linear map
\begin{equation}\label{tilde D_j as shifted linear map}
\xgr {D}_j:\xgr^i\mV_j[\uparrow]\ra\xgr^{i-1}\mV_{j+1}[\uparrow].
\end{equation}

\bigskip
Let $\mV$  be a $\gP$-module. Then by (\ref{inclusion of jets}),  we can view for  the vector space  $gr^r\mV$ as a linear subspace of $\xgr^r\mV$. We put for each $a\in S^k$  
\begin{equation}\label{shifted grading I}
gr^i\mV_a[\uparrow]:=gr^{i-q(a)}\mV_a
\end{equation}
and
\begin{equation}\label{shifted grading II}
gr^i\mV_j[\uparrow]=\bigoplus_{a\in S^k_j}gr^i\mV_{a}[\uparrow]
\end{equation}
so that $gr^i\mV_a[\uparrow]\subset \xgr^i\mV_a[\uparrow]$ and $gr^i\mV_j[\uparrow]\subset \xgr^i\mV_j[\uparrow]$. Then we have the following proposition.

\begin{prop}\label{thm formal exactness}
 The $k$-Dirac complex induces for each $i+j\ge0$ a long exact sequence
 \begin{equation}\label{exact complex of weighted jets}
 \xgr^{i+j}\mV_0[\uparrow]\xrightarrow{\xgr {D}_0}\xgr^{i+j-1}\mV_1[\uparrow]\ra\dots\ra\xgr^{i}\mV_j[\uparrow]\xrightarrow{\xgr {D}_j}\xgr^{i-1}\mV_{j+1}[\uparrow]\ra\dots
\end{equation}
This sequence contains a long exact subsequence
\begin{equation}\label{exact subcomplex of jets}
 gr^{i+j}\mV_0[\uparrow]\ra gr^{i+j-1}\mV_1[\uparrow]\ra\dots\ra gr^{i}\mV_j[\uparrow]\ra gr^{i-1}\mV_{j+1}[\uparrow]\ra\dots
\end{equation}
\end{prop}
\begin{proof}
 This follows from \cite[Theorem 7.14]{S} and the discussion from Section \ref{section embedding}.
\end{proof}

\subsection{The exactness of the $k$-Dirac complex in the sense of formal power series}\label{section exactness of k-Dirac complex with formal power series}
By the discussion that is given at the end of Section \ref{section affine subset}, the $k$-Dirac complex induces  a complex
\begin{equation}\label{k-Dirac complex over affine set}
\cC^\infty(\G_-,\mV_0)\xrightarrow{D_0}\cC^\infty(\G_-,\mV_1)\ra\dots\ra\cC^\infty(\G_-,\mV_j)\xrightarrow{  {D}_j}\cC^\infty(\G_-,\mV_{j+1})\ra\dots
 \end{equation}
of linear and $\G_-$-invariant differential operators with polynomial coefficients. We will now show that this complex  is exact with formal power series at any point $x\in\G_-$. Since each operator in the complex is invariant, it is certainly enough to show this at the origin $e\in\G_{-2}$. Here we will exploit the fact that with respect to the shifted grading introduced above each operator $D_j$ is  homogeneous of degree -1.

\bigskip

Recall (\ref{polynomial over affine set}) that the space of polynomials $S(\lap_+)$ over the affine set $\G_-$ has a natural grading $S(\lap_+)=\bigoplus_{r\ge0}\xS^r\lap_+$. Thus,  $\xS^r\lap_+\otimes\mV$ is the vector space of   functions $\G_-\ra\mV$ whose components  are (with respect to a fixed basis of $\mV)$ polynomials from $\xS^r\lap_+$. On the other hand by (\ref{isom weighted jets}), we can view $\xS^r\lap_+\otimes\mV$  for any irreducible $\gP$-module as the vector space $\xgr^r\mV$ and we will do that without further comment. Hence, there is a canonical isomorphism $\xS^r\lap_+\otimes\mV\ra\xgr^r\mV$ which appears  in the commutative diagram (\ref{commutative diagram with polynomials and weighted jets}) as the left vertical arrow.
\medskip

Let $\tilde\Phi(\mV_j)$ be the vector space of formal power series centered at $e\in\G_{-2}$ with values in the vector space $\mV_j$. So  $\Psi\in\tilde\Phi(\mV_j)$ can be written in a unique way as a formal sum $\sum_{i\ge 0}\Psi_i$ where $\Psi_i\in\xgr^i\mV_j[\uparrow]$. We obviously have  $D_j\Psi=\sum_{i\ge 0}D_j\Psi_i$.

\begin{thm}\label{thm exactnes with formal power series}
The sequence 
\begin{equation}\label{exact complex with formal power series}
  \tilde\Phi(\mV_0)\xrightarrow{D_0}\tilde\Phi(\mV_1)\ra\dots\ra\tilde\Phi(\mV_j)\xrightarrow{ {D}_j}\tilde\Phi(\mV_{j+1})\ra\dots
 \end{equation}
is exact.
\end{thm}
\begin{proof}
Let $\Psi=\sum_{i\ge 0}\Psi_i$ be as above. Then $\Psi_i$ is a $\mV_j$-valued polynomial function on $\G_-$ which (as explained above) we  view as an element of $\xgr^i\mV_j[\uparrow]$. Now $D_j\Psi_i$ is a $\mV_{j+1}$-valued polynomial on $\G_-$ which by the commutativity of the diagram (\ref{commutative diagram with polynomials and weighted jets}) and  (\ref{tilde D_j as shifted linear map}) belongs to $\xgr^{i-1}\mV_{j+1}$.
Now we can easily complete the proof. We see that ${D}_j\Psi=0$ if, and only if $ {D}_j\Psi_i=0$ for each $i\ge0$.  Now recall Proposition \ref{thm formal exactness} that  the complex (\ref{exact complex of weighted jets}) is exact for each $i+j\ge0$. This immediately implies that also the complex (\ref{exact complex with formal power series}) is exact.
\end{proof}

Let us now  recall a formula for the first operator $D_0$.  We have  $\mV_0\cong\Sp$ and  $\mV_1\cong\C^k\otimes\Sp$ as vector spaces and we view  $\mV_1$ as the vector space of $k$-tuples of spinors as in (\ref{k-Dirac operator}). It is shown in \cite{TS}  that 
\begin{equation}\label{k-Dirac operator over affine set}
{D}_0 f=\sum_{\alpha=1}^{2n}(\varepsilon_\alpha.L_{e^1\otimes\varepsilon_\alpha}f,\ \dots,\ \varepsilon_\alpha.L_{e^k\otimes\varepsilon_\alpha}f)
\end{equation}
where  $f\in\cC^\infty(\G_-,\Sp)$ and  $\varepsilon_\alpha.\in\mathrm{End}(\Sp)$ is the usual action of  $\varepsilon_\alpha$ on $\Sp$. 

 \subsection{The resolution of the  $k$-Dirac operator}\label{section resolution}
By Proposition \ref{thm descending}, the complex (\ref{k-Dirac complex over affine set}) descends to a complex
\begin{equation}\label{descended complex}
\cC^\infty(U,\mV_0)\xrightarrow{ \underline D_0}\cC^\infty(U,\mV_1)\xrightarrow{ \underline D_1}\dots\ra\cC^\infty(U,\mV_j)\xrightarrow{ \underline D_j}\cC^\infty(U,\mV_{j+1})\ra\dots
\end{equation}
of linear and  constant coefficient differential operators on $M(2n,k,\R)\cong U=\G_{-2}\setminus\G_-$. From the proof of Proposition \ref{thm descending} and (\ref{k-Dirac operator over affine set}), it immediately follows that the first operator $D_0$  is the $k$-Dirac operator (\ref{k-Dirac operator}). It remains to show that the sequence is locally exact.

\medskip

We will use the following notation. We  denote by $x_0:=\G_{-2}e$ the origin of $U$ and by $\Phi(\mV_j)$  the vector space of  formal power series centered at the point $x_0$ with values in the vector space $\mV_j$. An element  $\Psi\in\Phi(\mV_j)$ can be written in a unique way as an infinite formal sum $\sum_{i\ge 0}\Psi_i$ where $\Psi_i\in gr^i\mV_j[\uparrow]$. Here we use the fact that for any vector space $\mV$ the  space $gr^r\mV$ is naturally isomorphic  to the vector space of functions $U\ra\mV$ whose components are homogeneous polynomials of degree $r$.

\begin{thm}\label{the main theorem}
Assume that  $n\ge k\ge2$. Then the complex (\ref{descended complex}) is a locally exact sequence of linear and constant coefficient differential operators which starts with the $k$-Dirac operator.
 \end{thm}
\begin{proof}
By Theorem a. from \cite{Na}, it is enough to show that the complex (\ref{descended complex}) is exact with formal power series at the origin $x_0$, i.e. we need to show that the induced complex
\begin{equation}\label{exact complex with formal power series II}
\Phi(\mV_0)\xrightarrow{ \underline D_0} \Phi(\mV_1)\ra\dots\ra \Phi(\mV_j)\xrightarrow{\underline D_j} \Phi(\mV_{j+1})\ra\dots
\end{equation}
is exact.
Let $\Psi=\sum_{i\ge 0}\Psi_i$ be as above.  Arguing as in the proof of Theorem \ref{thm exactnes with formal power series} and using the commutativity of (\ref{commutative diagram with jets and weighted jets}), one concludes that $\underline D_j\Psi_i\in gr^{i-1}\mV_j[\uparrow]$. Since  $\underline D_j\Psi=\sum_{i\ge0}\underline D_j\Psi_i$, we see that $\underline D_j\Psi=0$ if, and only if $\underline D_j\Psi_i=0$ for each $i\ge0$. The exactness of (\ref{exact complex with formal power series II}) then follows from the fact that  the complex (\ref{exact subcomplex of jets}) is exact for each $i+j\ge0$.
\end{proof}

\def\bibname{Bibliography}

\end{document}